\documentclass[a4paper,12pt]{amsart}

\usepackage{amssymb}
\usepackage{latexsym}
\usepackage{amsmath}
\usepackage{amscd}
\usepackage{graphicx}
\usepackage{palatino}

\title[Torelli group, Johnson kernel and 
invariants of homology spheres]
{
Torelli group, Johnson kernel and 
invariants of homology spheres}

\author{Shigeyuki Morita}
\address{Graduate School of Mathematical Sciences, 
The University of Tokyo, 
3-8-1 Komaba, 
Meguro-ku, Tokyo, 153-8914, Japan}
\email{morita@ms.u-tokyo.ac.jp}
\author{Takuya Sakasai}
\address{Graduate School of Mathematical Sciences, 
The University of Tokyo, 
3-8-1 Komaba, 
Meguro-ku, Tokyo, 153-8914, Japan}
\email{sakasai@ms.u-tokyo.ac.jp}
\author{Masaaki Suzuki}
\address{Department of Frontier Media Science, 
Meiji University, 
4-21-1 Nakano, Nakano-ku, Tokyo, 164-8525, Japan}
\email{macky@fms.meiji.ac.jp}

\thanks{
The authors were partially supported by KAKENHI 
(No.15H03618, No.15H03619, No.16H03931, No. 16K05159 and No. 19G01785), 
Japan Society for the Promotion of Science, 
Japan.
}


\newtheorem{thm}{Theorem}[section]
\newtheorem{prop}[thm]{Proposition}

\newtheorem{cor}[thm]{Corollary}
\theoremstyle{definition}

\newtheorem{remark}[thm]{Remark}
\newtheorem{problem}[thm]{Problem}
\newtheorem{conj}[thm]{Conjecture}

\begin{document}

\newcommand{\Mg}{\mathcal{M}_g}
\newcommand{\Mgp}{\mathcal{M}_{g,\ast}}
\newcommand{\Mgb}{\mathcal{M}_{g,1}}

\newcommand{\hg}{\mathfrak{h}_{g,1}}
\newcommand{\ag}{\mathfrak{a}_g}
\newcommand{\Ln}{\mathcal{L}_n}

\newcommand{\Sg}{\Sigma_g}
\newcommand{\Sgb}{\Sigma_{g,1}}
\newcommand{\la}{\lambda}

\newcommand{\Symp}[1]{Sp(2g,\mathbb{#1})}
\newcommand{\symp}[1]{\mathfrak{sp}(2g,\mathbb{#1})}
\newcommand{\gl}[1]{\mathfrak{gl}(n,\mathbb{#1})}

\newcommand{\At}[1]{\mathcal{A}_{#1}^t (H)}
\newcommand{\Hq}{H_{\mathbb{Q}}}

\newcommand{\Ker}{\mathop{\mathrm{Ker}}\nolimits}
\newcommand{\Hom}{\mathop{\mathrm{Hom}}\nolimits}
\renewcommand{\Im}{\mathop{\mathrm{Im}}\nolimits}

\newcommand{\Der}{\mathop{\mathrm{Der}}\nolimits}
\newcommand{\Out}{\mathop{\mathrm{Out}}\nolimits}
\newcommand{\Aut}{\mathop{\mathrm{Aut}}\nolimits}
\newcommand{\Q}{\mathbb{Q}}
\newcommand{\Z}{\mathbb{Z}}
\newcommand{\R}{\mathbb{R}}
\newcommand{\C}{\mathbb{C}}

\begin{abstract}
In the late $1980$'s, 
it was shown 
that the Casson invariant
appears in the difference between the two filtrations of the Torelli group:
the lower central series and the Johnson filtration, 
and 
its core part was identified with the secondary characteristic class
$d_1$ associated with the fact that
the first $\mathrm{MMM}$ class vanishes on the Torelli group
(however it turned out that Johnson 
proved the former part 
highly likely
prior to the above,
see Remark \ref{rem:j}).
This secondary class $d_1$ is a rational generator of $H^1(\mathcal{K}_g;\Z)^{\mathcal{M}_g}\cong\Z$ where $\mathcal{K}_g$ denotes
the Johnson subgroup of the  mapping class group $\mathcal{M}_g$.
Hain proved,
as a particular case of his fundamental result, that
this is the only difference in degree $2$. In this paper, we prove that
no other invariant than the above gives rise to new rational difference between the two filtrations
up to degree $6$.
We apply this to determine $H_1(\mathcal{K}_g;\Q)$ explicitly
by computing the description given by Dimca, Hain and Papadima.
We also show that any finite type rational invariant of homology $3$-spheres
of degrees up to $6$,
including the second and the third Ohtsuki invariants,
can be expressed by 
$d_1$ and lifts of Johnson homomorphisms.

\end{abstract}

\renewcommand\baselinestretch{1.1}
\setlength{\baselineskip}{16pt}

\newcounter{fig}
\setcounter{fig}{0}

\maketitle

\section{Introduction and statements of the main results}\label{sec:intro}
Let $\mathcal{M}_g$ be the mapping class group of a closed oriented 
surface $\Sigma_g$ of genus $g$ and let $\mathcal{I}_g\subset \mathcal{M}_g$ be the Torelli
subgroup. Namely, it is the subgroup of $\mathcal{M}_g$ consisting of all the elements
which act on the homology $H:=H_1(\Sigma_g;\Z)$ trivially.

There exist two filtrations of the Torelli group. One is the lower central series
which we denote by
$\mathcal{I}_g(k)\ (k=1,2,\ldots)$ where $\mathcal{I}_g(1)=\mathcal{I}_g$,
$\mathcal{I}_g(2)=[\mathcal{I}_g, \mathcal{I}_g]$ and
$\mathcal{I}_g(k+1)=[\mathcal{I}_g(k), \mathcal{I}_g]$ for $k\geq 1$.
The other is called the Johnson filtration 
$\mathcal{M}_g(k)\ (k=1,2,\ldots)$ of the mapping class group
where $\mathcal{M}_g(k)$ is defined to be the kernel of the natural
homomorphism
$$
\rho_k: \mathcal{M}_g\rightarrow \mathrm{Out} (N_k(\pi_1\Sigma_g)).
$$
Here $N_k(\pi_1\Sigma_g)$ denotes the $k$-th nilpotent 
quotient of the fundamental group of $\Sigma_g$ and 
$\mathrm{Out} (N_k(\pi_1\Sigma_g))$ denotes its outer automorphism group.
$\mathcal{M}_g(1)$ is nothing other than the Torelli group $\mathcal{I}_g$
so that $\mathcal{M}_g(k)\ (k=1,2,\ldots)$ is a filtration of $\mathcal{I}_g$.
This filtration was originally introduced by Johnson \cite{johnson}
for the case of a genus $g$ surface with one boundary component.
The above is the one adapted to the case of a closed surface
(see \cite{morita99} for details).
It can be shown that $\mathcal{I}_g(k)\subset \mathcal{M}_g(k)$ for all $k\geq 1$.
Johnson showed in \cite{j3} that $\mathcal{I}_g(2)$ is a finite index subgroup of $\mathcal{M}_g(2)$
and asked whether this will continue to hold for the pair $\mathcal{I}_g(k)\subset \mathcal{M}_g(k)\ (k\geq 3)$.
He also showed in \cite{j2} that $\mathcal{M}_g(2)$ is equal to the subgroup $\mathcal{K}_g$,
called the Johnson subgroup or Johnson kernel, 
consisting of all the Dehn twists along separating simple closed curves on $\Sigma_g$.

The above question was answered negatively for the case $k=3$ in \cite{morita87}\cite{morita89}
(however,  it turned out that Johnson proved this fact
highly likely prior to the above
, see Remark \ref{rem:j} below).
More precisely, 
a homomorphism $d_1: \mathcal{K}_g\rightarrow \Z$ was constructed 
which is non-trivial on $\mathcal{M}_g(3)$ while it vanishes on $\mathcal{I}_g(3)$
so that the index of the pair $\mathcal{I}_g(3)\subset \mathcal{M}_g(3)$ was proved to be infinite.
Furthermore it was shown in \cite{morita91} that there exists an isomorphism
$$
H^1(\mathcal{K}_g;\Z)^{\mathcal{M}_g}\cong \Z\quad (g\geq 2)
$$
where the homomorphism $d_1$ serves as a rational generator.
It is characterized by the fact that its value on a separating simple closed curve
on $\Sigma_g$ of type $(h,g-h)$ is $h(g-h)$ up to non-zero constants.
This homomorphism was defined as the secondary characteristic class associated with the
fact that the first $\mathrm{MMM}$ class, which is an element of $H^2(\mathcal{M}_g;\Z)$, vanishes in
$H^2(\mathcal{I}_g;\Z)$. It was also interpreted as a manifestation of the Casson invariant $\lambda$, 
which is an invariant defined for homology $3$-spheres, in the structure of the Torelli group.

\begin{remark}
In a note of Johnson
opened as \cite{jnote},
he studied the influence of the Casson invariant on the structure of the Johnson kernel.
It turned out that he proved a negative answer for the case $k=3$ 
to his question mentioned above  
and also obtained a large part of those results of \cite{morita87}\cite{morita89} concerning the
above topic.
Although his note is undated, it seems 
highly likely that
his work was done 
prior 
to what was published in the 
above papers. 
This is very surprising. On the other hand, 
the main result of the above papers is not covered, which expresses the Casson invariant as the
secondary invariant associated to the fact that the first MMM class vanishes on the Torelli group.
\label{rem:j}
\end{remark}

Now let us consider the following two graded Lie algebras
\begin{align*}
\mathrm{Gr}\,\mathfrak{t}_g&=\bigoplus_{k=1}^\infty \mathfrak{t}_g(k),\quad \mathfrak{t}_g(k)
=(\mathcal{I}_g(k)/\mathcal{I}_g(k+1))\otimes\Q\\
\mathfrak{m}_g&=\bigoplus_{k=1}^\infty \mathfrak{m}_g(k),\quad 
\mathfrak{m}_g(k)=(\mathcal{M}_g(k)/\mathcal{M}_g(k+1))\otimes\Q
\end{align*}
associated to the above two filtrations of the Torelli group.
Here $\mathfrak{t}_g$ denotes the Malcev Lie algebra of the Torelli group
and $\mathrm{Gr}\,\mathfrak{t}_g$ denotes its associated graded Lie algebra.
Hain \cite{hain} obtained fundamental results about the structure of these Lie algebras.
He gave an explicit finite presentation of them (see Theorem \ref{th:hain} below)
which implies that $\mathrm{Ker}(\mathfrak{t}_g(2)\rightarrow \mathfrak{m}_g(2))\cong\Q$.
Furthermore, he proved that
the natural homomorphism
\begin{equation}
\mathfrak{t}_g(k)\rightarrow \mathfrak{m}_g(k)
\label{eq:nh}
\end{equation}
is surjective for any $k$
which implies 
that the index of the pair $\mathcal{I}_g(k)\subset \mathcal{M}_g(k)$ 
remains infinite for any $k\geq 4$ extending the above mentioned result.
He also showed that all the higher Massey products
of the Torelli group vanish for $g\geq 4$. 

On the other hand, Ohtsuki \cite{ohtsukip} defined a series of invariants
$\lambda_k$ for homology $3$-spheres the first one being the Casson invariant $\lambda$.
He also initiated a theory of finite type invariants for homology $3$-spheres in \cite{ohtsuki}.
Then Garoufalidis and Levine \cite{gl} studied the relation between this theory and the structure
of the Torelli group extending the case of the Casson invariant mentioned above extensively.

In these situations, it would be natural to ask whether there exists any other
difference between the two filtrations of the Torelli group than the Casson invariant,
in particular whether any finite type rational invariant of homology $3$-spheres of degree
greater than $2$ appears there or not. This is equivalent to asking whether the
natural homomorphism \eqref{eq:nh} is an isomorphism for $k=3,4,\ldots$ or not.

Now it was proved in \cite{morita99} that $\mathfrak{t}_g(3)\cong \mathfrak{m}_g(3)$.
The main theorem of the present paper is the following.
\begin{thm}
For any $k=4,5,6,$ we have
$$
\mathfrak{t}_g(k)\cong \mathfrak{m}_g(k).
$$
\label{th:main}
\end{thm}

As a corollary to the above theorem, we obtain the cases $k=5,6,7$ of the following result.
The case $k=3$ follows from Hain's theorem (Theorem \ref{th:hain}) 
combined with a result of \cite{morita89} and
the case $k=4$ follows from a result of \cite{morita99} mentioned above.

\begin{cor}
For any $k=3,4,5,6,7$, the $k$-th group
$
\mathcal{I}_g(k)
$
in the lower central series of the Torelli group
is a finite index subgroup of the kernel of the
non-trivial homomorphism
$$
d_1: \mathcal{M}_g(k)\rightarrow \Z.
$$
\label{cor:fi}
\end{cor}

Recall here that Johnson \cite{j3} proved that $\mathcal{I}_g(2)=[\mathcal{I}_g,\mathcal{I}_g]$
is a finite index subgroup of $\mathcal{M}_g(2)=\mathcal{K}_g$.

Next we present two applications of Theorem \ref{th:main}.
First, we give the explicit form of the rational abelianization 
$H_1(\mathcal{K}_g;\Q)$ of the Johnson subgroup.  Dimca and Papadima \cite{dp} proved
that $H_1(\mathcal{K}_g;\Q)$ is finite dimensional for $g\geq 4$. Then Dimca, Hain and Papadima
\cite{dhp} gave a description of it. However they did not give the final explicit form.
Here we compute their description by combining the case $k=4$ of Theorem \ref{th:main}
and former results concerning the Johnson homomorphisms
to obtain the following result.

\begin{thm}
The secondary class $d_1$ together with the refinement $\tilde{\tau}_2$ of the second 
Johnson homomorphism gives the following isomorphism for $g\geq 6$.
$$
H_1(\mathcal{K}_g;\Q)\cong \Q\oplus [2^2]\oplus [31^2].
$$
\label{th:kab}
\end{thm}
Here and henceforth, for a given Young diagram
$\lambda=[\lambda_1\cdots \lambda_h]$, we denote the 
irreducible representation of $\mathrm{Sp}(2g,\Q)$ corresponding to
$\lambda$ simply by $[\lambda_1\cdots \lambda_h]$. As for the refinements $\tilde{\tau}_k$ of
Johnson homomorphisms, see Section \ref{sec:ft} for details.

By making use of recent remarkable results of Ershov-He \cite{eh} and Church-Ershov-Putman \cite{cep}, we obtain
the following.
\begin{cor}
$\mathrm{(i)}\ $
The two subgroups $[\mathcal{K}_g,\mathcal{K}_g]$ and $\mathcal{I}_g(4)$ of the Torelli group
$\mathcal{I}_g$ are commensurable for $g\geq 6$.

$\mathrm{(ii)}\ $ $[\mathcal{K}_g,\mathcal{K}_g]$ is finitely generated for $g\geq 7$.
\label{cor:ki}
\end{cor}

\begin{remark}
Johnson \cite{j3} determined the abelianization $H_1(\mathcal{I}_g;\Z)$ of the Torelli group completely, 
where the Birman-Craggs homomorphisms introduced in \cite{bc} played an
essential role in describing its torsion part. Although some of the Birman-Craggs homomorphisms 
restrict non-trivially on $\mathcal{K}_g$, they are mod $2$ reductions of integral ones
because the Casson invariant defines an integer valued homomorphism on $\mathcal{K}_g$.
Therefore, no $2$-torsion class in $H_1(\mathcal{I}_g;\Z)$ can be lifted to $H_1(\mathcal{K}_g;\Z)$
as a torsion class. Thus,
at present, there is no known information about the
torsion part of $H_1(\mathcal{K}_g;\Z)$. It should be an important problem to determine it.
\end{remark}

Another application of our main theorem is the following.
\begin{thm}
Any finite type rational invariant of homology $3$-spheres of degrees $4$ and $6$, 
including the Ohtsuki invariants
$\lambda_2$ and $\lambda_3$, 
can be expressed
by $d_1$ and (lifts of) Johnson homomorphisms.
\label{th:fti}
\end{thm}

In Section \ref{sec:ft}, we give more detailed statements Theorem \ref{th:type4} and Theorem \ref{th:type6}.

Based on the above result, we would like to propose the following conjecture
(see Problem 6.2 of \cite{morita99}).

\begin{conj} For any $k\not=2$, the equality
$
\mathfrak{t}_g(k)\cong \mathfrak{m}_g(k)
$
holds so that
$$
\mathrm{Ker}(\mathrm{Gr}\,\mathfrak{t}_g\twoheadrightarrow \mathfrak{m}_g)\cong \Q.
$$
\end{conj}

\begin{remark}\label{rem:aftercor16}
This is equivalent to the statement that Corollary \ref{cor:fi} continues to hold for all $k \geq 3$. 
\end{remark}

In the context of characteristic classes of the mapping class group, the above conjecture
can be translated as follows.
\begin{conj}[another formulation]
The Lie algebra
$\mathfrak{t}_g$ is isomorphic to the completion of the central extension of $\mathfrak{m}_g$
associated to the infinitesimal first $\mathrm{MMM}$ class defined in
$H^2(\mathfrak{m}_g)$.
\end{conj}

Hain \cite{hain} considered the relative completion of the mapping class group 
with respect to the classical homomorphism $\mathcal{M}_g\rightarrow\mathrm{Sp}(2g,\Z)$.
He proved that the kernel of the natural surjective homomorphism $\mathfrak{t}_g\rightarrow \mathfrak{u}_g$
is isomorphic to $\Q$ where $\mathfrak{u}_g$ denotes the graded Lie algebra associated to
the Lie algebra of his relative completion. In this terminology,
the above conjecture is also equivalent to saying that the natural homomorphism
$$
\mathrm{Gr}\,\mathfrak{u}_g\rightarrow \mathfrak{m}_g,
$$
which exists because of the universality of $\mathfrak{u}_g$,
is an isomorphism.

The content of the present paper is roughly as follows. In Section $2$, we relate the difference
between the two filtrations of the Torelli group to the second homology group
$H_2(\mathfrak{m}_g)$
of the Lie algebra $\mathfrak{m}_g$. In Section $3$, we explain the method of proving the main Theorem \ref{th:main}.
Then in Sections $4, 5$ and  $6$, we prove the vanishing of the weight $4$, $5$ and $6$
parts of $H_2(\mathfrak{m}_g)$, respectively. Finally in Section $7$, we prove the main results.

In this paper, whenever we mention groups like
$\mathcal{M}_g, \mathcal{I}_g,\mathcal{M}_g(k),\mathcal{I}_g(k)$,
modules like 
$\mathfrak{t}_g(k), \mathfrak{m}_g(k)$
and also homomorhpsims like $\tau_g(k)$, which depend on
the genus $g$, we always assume that it is sufficiently large,
more precisely in a stable range with respect to
the property we consider, unless we describe the range of $g$
explicitly.

\section{The second homology groups of the Lie algebras $\mathfrak{t}_g, \mathfrak{m}_g$}\label{sec:h2}

In this section, we reduce the problem of determining the kernel of
the surjective homomorphism 
$\mathrm{Gr}\,\mathfrak{t}_g\twoheadrightarrow \mathfrak{m}_g$
to the computation of the second homology group $H_2(\mathfrak{m}_g)$ of the Lie algebra $\mathfrak{m}_g$.
Let us denote $\mathrm{Ker}(\mathrm{Gr}\,\mathfrak{t}_g\rightarrow \mathfrak{m}_g)$ by $\mathfrak{i}_g$ so that
we have a short exact sequence
\begin{equation}
0\rightarrow \mathfrak{i}_g
\rightarrow \mathrm{Gr}\,\mathfrak{t}_g\rightarrow\mathfrak{m}_g\rightarrow 0
\label{eq:itm}
\end{equation}
of the three graded Lie algebras
$$
\mathfrak{i}_g=\bigoplus_{k=1}^\infty\mathfrak{i}_g(k),\
\mathrm{Gr}\,\mathfrak{t}_g=\bigoplus_{k=1}^\infty\mathfrak{t}_g(k),\
\mathfrak{m}_g=\bigoplus_{k=1}^\infty\mathfrak{m}_g(k).
$$

Now it is a classical result of Johnson \cite{ja}\cite{j3} that 
$\mathfrak{t}_g(1)\cong \mathfrak{m}_g(1)\cong\wedge^3 H_\Q/H_\Q$
and hence $\mathfrak{i}_g(1)=0$. The module $\wedge^3 H_\Q/H_\Q$ is an
irreducible representation of $\mathrm{Sp}(2g,\Q)$ corresponding to the
Young diagram $[1^3]$. 

Hain proved the following fundamental result.

\begin{thm}[Hain \cite{hain}]
The Lie algebra $\mathfrak{t}_g\ (g\geq 6)$ is isomorphic to its associated
graded $\mathrm{Gr}\,\mathfrak{t}_g$ which has presentation
$$
\mathrm{Gr}\,\mathfrak{t}_g=\mathcal{L}(
[1^3])/\langle[1^6]+[1^4]+[1^2]+[2^21^2]\rangle
$$
where $\mathcal{L}([1^3])$ denotes the free Lie algebra generated by $[1^3]$ and
$\langle[1^6]+[1^4]+[1^2]+[2^21^2]\rangle$ is the ideal generated by
$$
[1^6]+[1^4]+[1^2]+[2^21^2]\subset \wedge^2[1^3].
$$
\label{th:hain}
\end{thm}

Here we recall a few facts about the homology of graded Lie algebras briefly.
Let us consider a graded Lie algebra
$$
\mathfrak{g}=\bigoplus_{k=1}^\infty \mathfrak{g}(k)
$$
satisfying the condition that
$
H_1(\mathfrak{g})\cong \mathfrak{g}(1).
$
Namely, we assume that $\mathfrak{g}$ is generated by the degree $1$ part $\mathfrak{g}(1)$
as a Lie algebra. Both the Lie algebras $\mathfrak{t}_g, \mathfrak{m}_g$ satisfy this condition.
The homology group of a graded Lie algebra is bigraded, where the first grading is
the usual homology degree while the second grading is induced from the
grading of the graded Lie algebra. We call the latter grading by weight, denoted simply by the
subscript $w$.
In particular, the second homology group has the following direct sum decomposition.
$$
H_2(\mathfrak{g})=\bigoplus_{w=2}^\infty H_2(\mathfrak{g})_w
$$
where
\begin{equation}
H_2(\mathfrak{g})_w=
\frac{\mathrm{Ker}\left(\partial: (\wedge^{2}\mathfrak{g})_w
\twoheadrightarrow \mathfrak{g}(w)\right)}{\mathrm{Im}\left(\partial: (\wedge^3\mathfrak{g})_w\rightarrow 
(\wedge^{2}\mathfrak{g})_w\right)}.
\label{eq:hw}
\end{equation}

In this terminology, the following is an immediate consequence of 
Theorem \ref{th:hain}.

\begin{cor}\label{cor:22}
$H_2(\mathrm{Gr}\,\mathfrak{t}_g)_2\cong [1^6]+[1^4]+[1^2]+[2^21^2]$ and
for any $w\geq 3$, $H_2(\mathrm{Gr}\,\mathfrak{t}_g)_w=0$.
\label{cor:hain}
\end{cor}

Now we consider the short exact sequence \eqref{eq:itm}.
\begin{prop}
The following short exact sequence
$$
0\rightarrow H_2(\mathrm{Gr}\,\mathfrak{t}_g)_2\rightarrow H_2(\mathfrak{m}_g)_2\rightarrow 
(H_1(\mathfrak{i}_g)_{\mathfrak{m}_g})_2\cong\Q
\rightarrow
0
$$
holds and for any $w\geq 3$
$$
H_2(\mathfrak{m}_g)_w\cong (H_1(\mathfrak{i}_g)_{\mathfrak{m}_g})_w.
$$
\label{prop:mi}
\end{prop}
\begin{proof}
The $5$-term exact sequence of the Lie algebra extension \eqref{eq:itm}
is given by
$$
H_2(\mathrm{Gr}\,\mathfrak{t}_g)\overset{\text{Cor. \ref{cor:22}}}{=}H_2(\mathrm{Gr}\,\mathfrak{t}_g)_2
\rightarrow H_2(\mathfrak{m}_g)\rightarrow 
H_1(\mathfrak{i}_g)_{\mathfrak{m}_g}
\rightarrow
H_1(\mathrm{Gr}\,\mathfrak{t}_g)
\overset{\cong}{\rightarrow} H_1(\mathfrak{m}_g).
$$
The result follows from this.
\end{proof}

\begin{cor}
$$
H_2(\mathfrak{m}_g)_3=0 ,\quad
\mathfrak{i}_g(4)\cong H_2(\mathfrak{m}_g)_4
$$
and
for any $w\geq 4$, we have the following exact sequence.
$$
0\rightarrow  \sum_{k=3}^{w-1}\ [\mathfrak{i}_g(k), \mathfrak{t}_g(w-k)]
\rightarrow\mathfrak{i}_g(w)\rightarrow H_2(\mathfrak{m}_g)_w\rightarrow 0.
$$
\label{cor:h2}
\end{cor}
\begin{proof}
It was shown in \cite{morita99} (see also a related result of \cite{sakasai})
that $\mathfrak{t}_g(3)\cong\mathfrak{m}_g(3)$ and hence $\mathfrak{i}_g(3)=0$.
It follows that $H_2(\mathfrak{m}_g)_3\cong (H_1(\mathfrak{i}_g)_{\mathfrak{m}_g})_3=0$.
As for the case $w=4$, we have
$$
H_2(\mathfrak{m}_g)_4\cong (H_1(\mathfrak{i}_g)_{\mathfrak{m}_g})_4
=\mathfrak{i}_g(4)
$$
because $\mathfrak{i}_g(2)\cong\Q$ is contained in the center of $\mathfrak{t}_g$.
Finally, for any $w\geq 4$, we have
$$
H_2(\mathfrak{m}_g)_w\cong (H_1(\mathfrak{i}_g)_{\mathfrak{m}_g})_w
\cong \mathfrak{i}_g(w)/ \sum_{k=3}^{w-1}\ [\mathfrak{i}_g(k), \mathfrak{t}_g(w-k)].
$$
This completes the proof.
\end{proof}

In general, we have the following

\begin{prop}
Assume $\mathfrak{i}_g(k)=0\ $ for $k=3,4,\ldots,w-1\ (w\geq 4)$. Then we have
$$
\mathfrak{i}_g(w)\cong H_2(\mathfrak{m}_g)_w.
$$
\label{prop:iw}
\end{prop}
\begin{proof}
This is an immediate consequence of Corollary \ref{cor:h2}.
\end{proof}

\section{Method of proving 
Theorem \ref{th:main}}\label{sec:method}

To prove Theorem \ref{th:main}, it is enough to show that
$$
H_2(\mathfrak{m}_g)_w=0 \ (w=4,5,6)
$$
by Proposition \ref{prop:iw}. We prove this in 
Sections $4,5,6$.
To do so, 
for technical reasons regarding computer computations, 
we consider the following variants of our Lie algebra. 

Let $\mathcal{M}_{g,1}$ be the mapping class group of a genus $g$ compact oriented surface
with one boundary component and let $\mathcal{I}_{g,1}$ be its Torelli subgroup.
Then we have the Johnson filtration $\{\mathcal{M}_{g,1}(k)\}_k$ for the former and the lower central
series $\{\mathcal{I}_{g,1}(k)\}_k$ for the latter. We set
\begin{align*}
\mathrm{Gr}\,\mathfrak{t}_{g,1}&=\bigoplus_{k=1}^\infty \mathfrak{t}_{g,1}(k),\quad \mathfrak{t}_{g,1}(k)
=(\mathcal{I}_{g,1}(k)/\mathcal{I}_{g,1}(k+1))\otimes\Q\\
\mathfrak{m}_{g,1}&=\bigoplus_{k=1}^\infty \mathfrak{m}_{g,1}(k),\quad 
\mathfrak{m}_{g,1}(k)=(\mathcal{M}_{g,1}(k)/\mathcal{M}_{g,1}(k+1))\otimes\Q
\end{align*}
where $\mathfrak{t}_{g,1}$ denotes the Malcev Lie algebra of $\mathcal{I}_{g,1}$.
Define $\mathfrak{i}_{g,1}=\mathrm{Ker}(\mathrm{Gr}\,\mathfrak{t}_{g,1}\twoheadrightarrow\mathfrak{m}_{g,1})$
so that we have a short exact sequence
\begin{equation}
0\rightarrow \mathfrak{i}_{g,1}\rightarrow \mathrm{Gr}\,\mathfrak{t}_{g,1}\rightarrow\mathfrak{m}_{g,1}\rightarrow 0
\label{eq:itm1}
\end{equation}
of graded Lie algebras, where the surjectivity of the natural homomorphism 
$\mathrm{Gr}\,\mathfrak{t}_{g,1}\rightarrow\mathfrak{m}_{g,1}$ is again due to Hain.

\begin{prop}
The natural homomorphism $\mathfrak{i}_{g,1}\rightarrow \mathfrak{i}_{g}$
is an isomorphism so that
$\mathfrak{i}_{g,1}(k)\cong \mathfrak{i}_{g}(k)$ for any $k$.
\label{prop:ii}
\end{prop}
\begin{proof}
Recall that the kernel of the natural surjection $\mathcal{I}_{g,1}\rightarrow \mathcal{I}_g$
is known to be isomorphic to $\pi_1 T_1\Sigma_g$ where $T_1\Sigma_g$ denotes the unit tangent bundle of
$\Sigma_g$. Thus we have a short exact sequence
$1\rightarrow \pi_1 T_1\Sigma_g\rightarrow \mathcal{I}_{g,1}\rightarrow \mathcal{I}_g\rightarrow 1$.
Let $\mathfrak{p}_{g,1}$ be the Malcev Lie algebra of $\pi_1 T_1\Sigma_g$
which is a central extension of the Malcev Lie algebra $\mathfrak{p}_g$ of $\pi_1\Sigma_g$. 
Hain \cite{hain} showed that this induces 
a short exact sequence 
$0\rightarrow \mathfrak{p}_{g,1}\rightarrow \mathfrak{t}_{g,1}\rightarrow \mathfrak{t}_{g}\rightarrow 0$
which in turn induces a short exact sequence of their associated graded
Lie algebras as depicted in 
\eqref{eq:cd}.
\begin{equation}
\begin{CD}
@. @.   0 @. 0@.\\
@. @. @VVV @VVV @.\\
@.   @. \mathfrak{i}_{g,1} @>{\cong}>> \mathfrak{i}_{g}
@.\\
@. @. @VVV @VVV @.\\
0 @>>> \mathrm{Gr}\,\mathfrak{p}_{g,1}   @>>> \mathrm{Gr}\,\mathfrak{t}_{g,1} @>>> \mathrm{Gr}\,\mathfrak{t}_{g}
@>>> 0\\
@. @| @VVV @VVV @.\\
0 @>>> \mathrm{Gr}\,\mathfrak{p}_{g,1} @>>> \mathfrak{m}_{g,1} @>>> 
\mathfrak{m}_{g} @>>> 0\\
@. @. @VVV @VVV @.\\
@.   @. 0 @. 0
\end{CD}
\label{eq:cd}
\end{equation}
On the other hand,
a result of Asada and Kaneko (and Labute) in \cite{ak} implies that
the kernel of the natural surjection $\mathfrak{m}_{g,1}\rightarrow\mathfrak{m}_{g}$
is isomorphic to $\mathrm{Gr}\,\mathfrak{p}_{g,1}$
(see \cite{morita99}\cite{mss4} for details).
Thus the two rows as well as the two columns of the commutative diagram \eqref{eq:cd} are
all short exact sequences. Then it is easy to see from this fact that the homomorphism
$\mathfrak{i}_{g,1}\rightarrow\mathfrak{i}_g$ is an isomorphism. This completes the proof.
\end{proof}

Now we consider the short exact sequence \eqref{eq:itm1}.
\begin{prop}
For any $w\geq 3$,
$$
H_2(\mathfrak{m}_{g,1})_w\cong (H_1(\mathfrak{i}_{g,1})_{\mathfrak{m}_{g,1}})_w.
$$
\label{prop:mi1}
\end{prop}
\begin{proof}
In \cite[Corollary 7.8]{hain}, Hain showed that $\mathrm{Gr}\,\mathfrak{t}_{g,1}$ is quadratically presented. 
This implies that $H_2(\mathrm{Gr}\,\mathfrak{t}_{g,1}) \cong H_2(\mathrm{Gr}\,\mathfrak{t}_{g,1})_2 $. 
Then the $5$-term exact sequence of the Lie algebra extension \eqref{eq:itm1} is given by
\begin{align*}
H_2(\mathrm{Gr}\,\mathfrak{t}_{g,1})=H_2(\mathrm{Gr}\,\mathfrak{t}_{g,1})_2
\rightarrow H_2(\mathfrak{m}_{g,1})\rightarrow 
H_1(\mathfrak{i}_{g,1})_{\mathfrak{m}_{g,1}}
\rightarrow
H_1(\mathrm{Gr}\,\mathfrak{t}_{g,1})
\overset{\cong}{\rightarrow} H_1(\mathfrak{m}_{g,1}).
\end{align*}
The result follows from this.
\end{proof}

\begin{prop}
The natural homomorphism
$$
H_2(\mathfrak{m}_{g,1})_w\rightarrow H_2(\mathfrak{m}_{g})_w
$$
is an isomorphism for all $w\geq 3$.
\label{prop:h2g1}
\end{prop}
\begin{proof}
By Proposition \ref{prop:mi1} and Proposition \ref{prop:mi},
the natural homomorphism
$
H_2(\mathfrak{m}_{g,1})_w\rightarrow H_2(\mathfrak{m}_{g})_w
$
can be identified with the natural homomorphism
$$
(H_1(\mathfrak{i}_{g,1})_{\mathfrak{m}_{g,1}})_w\rightarrow 
(H_1(\mathfrak{i}_{g})_{\mathfrak{m}_{g}})_w
$$
for any $w\geq 3$. Then it is easy to see that this is an isomorphism
by Proposition \ref{prop:ii}.
\end{proof}

As mentioned in the beginning of this section, 
to prove Theorem \ref{th:main}, it is enough to show that
$
H_2(\mathfrak{m}_g)_w=0 \ (w=4,5,6).
$
By the above Proposition \ref{prop:h2g1}, this is equivalent to
showing that
$$
H_2(\mathfrak{m}_{g,1})_w=0 \ (w=4,5,6).
$$
We consider this case of one boundary component which is best suited 
for our computer computation because of the following reason.
By virtue of the Johnson homomorphisms (see \cite{ja}\cite{johnson}\cite{morita93}), the Lie algebra
$\mathfrak{m}_{g,1}$ can be embedded into  
the Lie algebra $\mathfrak{h}_{g,1}$ consisting of all the symplectic derivations
of the free Lie algebra generated by $H_\Q=H_1(\Sigma_g,\Q)$
as a Lie subalgebra. More precisely, we have an embedding
$$
\tau_{g,1}=\bigoplus_{k=1}^\infty \tau_{g,1}(k):
\mathfrak{m}_{g,1}=\bigoplus_{k=1}^\infty
\mathfrak{m}_{g,1}(k)\rightarrow
\mathfrak{h}_{g,1}=\bigoplus_{k=1}^\infty
\mathfrak{h}_{g,1}(k)
\subset \bigoplus_{k=1}^\infty H_\Q^{\otimes (k+2)}.
$$
Thus each graded piece $\mathfrak{m}_{g,1}(k)$ can be identified with $\mathrm{Im}\,\tau_{g,1}(k)$,
which is a submodule of $H_\Q^{\otimes (k+2)}$, so that
we can make explicit computer computations on this space of tensors.

\section{Proof of $H_2(\mathfrak{m}_g)_4=0$}\label{sec:w4}
In this section, we prove the following.
\begin{prop}
$H_2(\mathfrak{m}_g)_4\cong H_2(\mathfrak{m}_{g,1})_4=0.$
\label{prop:w4}
\end{prop}

It is sufficient to show that $H_2(\mathfrak{m}_{g,1})_4=0$, by Proposition \ref{prop:h2g1}. 
By equality \eqref{eq:hw}, 
we have
$$
H_2(\mathfrak{m}_{g,1})_4=
\frac{\mathrm{Ker}\left((\mathfrak{m}_{g,1}(1)\otimes \mathfrak{m}_{g,1}(3))\oplus \wedge^2 \mathfrak{m}_{g,1}(2)
\twoheadrightarrow \mathfrak{m}_{g,1}(4)\right)}{\mathrm{Im}\left(\wedge^2\mathfrak{m}_{g,1}(1)\otimes \mathfrak{m}_{g,1}(2)\rightarrow 
(\mathfrak{m}_{g,1}(1)\otimes \mathfrak{m}_{g,1}(3))\oplus \wedge^2 \mathfrak{m}_{g,1}(2)\right)}.
$$
Here the boundary operator
\begin{equation*}
\partial : \wedge^2\mathfrak{m}_{g,1}(1)\otimes \mathfrak{m}_{g,1}(2)\rightarrow 
(\mathfrak{m}_{g,1}(1)\otimes \mathfrak{m}_{g,1}(3))\oplus \wedge^2 \mathfrak{m}_{g,1}(2)
\label{eq:b}
\end{equation*}
is given by
\begin{align*}
&\wedge^2\mathfrak{m}_{g,1}(1)\otimes \mathfrak{m}_{g,1}(2)\ni (u\land v)\otimes w\longmapsto\\
&(u\otimes [v,w]-v\otimes [u,w],-[u,v]\land w)\in (\mathfrak{m}_{g,1}(1)\otimes \mathfrak{m}_{g,1}(3))\oplus \wedge^2 \mathfrak{m}_{g,1}(2).
\end{align*}

In our paper \cite{mss4}, we gave the irreducible decompositions of 
$\mathfrak{m}_{g,1}(k)\cong \mathrm{Im}\,\tau_{g,1}(k)$
for all $k\leq 6$ as shown in Table \ref{tab:6}.

\begin{table}[h]
\caption{$\text{List of $
\mathfrak{m}_{g,1}(k)
$}$}
\begin{center}
\begin{tabular}{|l|l|}
\noalign{\hrule height0.8pt}
\hfil $k$ & $\mathfrak{m}_{g,1}(k)$  \\
\hline
$1$ &  $[1^3]\ [1]$ \\
\hline
$2$ & $[2^2]\ [1^2]\ [0]$ \\
\hline
$3$ &  $[31^2]\ [21]$ \\
\hline
$4$ & $[42][31^3][2^3]\ 2[31][21^2]\ 2[2]$ \\
\hline
$5$ & $[51^2][421][3^21][321^2][2^21^3]\ 2[41]2[32]2[31^2]2[2^21]2[21^3]\ [3]3[21]2[1^3]\ [1]$ \\
\hline
$6$ & $[62] [521][51^3][4^2][431]2[42^2][421^2][41^4]2[3^21^2][32^21][321^3][2^4]][2^21^4]$\\
{} & $3[51]3[42]4[41^2]3[3^2]7[321]3[31^3][2^3]5[2^21^2]2[21^4][1^6]$ \\
{} & $4[4]6[31]9[2^2]6[21^2]4[1^4]\ 3[2]6[1^2]\ 2[0]$ \\
\noalign{\hrule height0.8pt}
\end{tabular}
\end{center}
\label{tab:6}
\end{table}

By using this and applying our techniques described in \cite{mss2},
we determine the space 
$$
Z_2(4)=\mathrm{Ker}\left((\mathfrak{m}_{g,1}(1)\otimes \mathfrak{m}_{g,1}(3))\oplus 
\wedge^2 \mathfrak{m}_{g,1}(2)
\twoheadrightarrow \mathfrak{m}_{g,1}(4)\right)
$$
of $2$-cycles for the weight $4$ homology group
$H_2(\mathfrak{m}_{g,1})_4$ as shown in Table \ref{tab:z24}.
\begin{table}[h]
\caption{$\text{$\mathrm{Sp}$-irreducible decomposition of $Z_2(4)$}$}
\begin{center}
\begin{tabular}{|c|l|}
\noalign{\hrule height0.8pt}
$Z_2(4)$ & $[431][42^2][421^2][41^4]2[32^21][321^3][31^5]$ \\
{} & $[42]2[41^2][3^2]6[321]4[31^3][2^3]3[2^21^2]2[21^4]$\\
{} & $[4]5[31]5[2^2]7[21^2][1^4]\ 3[2]4[1^2]$\\
\hline
$\mathfrak{m}_{g,1}(1)\otimes \mathfrak{m}_{g,1}(3)$ &  
$[42^2][421^2][41^4][32^21][321^3][31^5]$ \\
{} & $[42]2[41^2]4[321]4[31^3][2^3]2[2^21^2]2[21^4]$ \\
{} & $[4]5[31]3[2^2]5[21^2][1^4]\ 3[2]2[1^2]$\\
\hline
$\wedge^2 \mathfrak{m}_{g,1}(2)$ & 
$[431][32^21]\ [42][3^2]2[321][31^3][2^3][2^21^2]\ 2[31]2[2^2]3[21^2]\ 2[2]2[1^2]$ \\
\hline
$\mathfrak{m}_{g,1}(4)$ &  $[42][31^3][2^3]\ 2[31][21^2]\ 2[2]$  \\
\noalign{\hrule height0.8pt}
\end{tabular}
\end{center}
\label{tab:z24}
\end{table}
Thus we can write 
$$
H_2(\mathfrak{m}_{g,1})_4\cong 
\mathrm{Coker}\, \left(
\wedge^2\mathfrak{m}_{g,1}(1)\otimes \mathfrak{m}_{g,1}(2)
\overset{\partial}{\rightarrow}
Z_2(4) \right).
$$

Now we show that the cokernel is trivial for each irreducible component. 
As an example, we pick up $[21^2]$. 
The multiplicity of $[21^2]$ in
$Z_2(4)$ is $7$, which is the biggest
(see Table \ref{tab:z24}). We have checked that all of these 
$[21^2]$ components are hit by the boundary operator 
$$
\partial : \wedge^2\mathfrak{m}_{g,1}(1)\otimes \mathfrak{m}_{g,1}(2)\rightarrow 
Z_2(4)\subset
(\mathfrak{m}_{g,1}(1)\otimes \mathfrak{m}_{g,1}(3))\oplus \wedge^2 \mathfrak{m}_{g,1}(2)
$$
as follows. 
First, 
$\mathfrak{m}_{g,1}(1)\otimes \mathfrak{m}_{g,1}(3)$ and $\wedge^2 \mathfrak{m}_{g,1}(2)$
are regarded 
as subspaces of $H_\Q^{\otimes 8}$ via the natural embeddings
\begin{align*}
\mathfrak{m}_{g,1}(1)\otimes \mathfrak{m}_{g,1}(3)
&\subset \mathfrak{h}_{g,1}(1)\otimes \mathfrak{h}_{g,1}(3)
\subset H_\Q^{\otimes 3}\otimes H_\Q^{\otimes 5}
\subset H_\Q^{\otimes 8}, \\
\wedge^2 \mathfrak{m}_{g,1}(2)
&\subset \wedge^2 \mathfrak{h}_{g,1}(2)
\subset \wedge^2 H_\Q^{\otimes 4}
\subset H_\Q^{\otimes 8}.
\end{align*}
Then we constructed $7$ $\mathrm{Sp}$-projections
$D_i: H_\Q^{\otimes 8}\rightarrow [21^2]\ (i=1,\ldots,7)$
which detect the $7$ copies of $[21^2]$ in $Z_2(4)$. 
Here we denote by $\mu_{(i_1,j_1)\cdots(i_k,j_k)}$ and $p_{(i_1,\ldots,i_k) \cdots (j_1,\ldots, j_l)}$ 
{\it multiple contractions} and {\it projections} respectively: 
\[
\mu_{(i_1,j_1)\cdots(i_k,j_k)} : H^{\otimes n}_\Q \to H^{\otimes (n-2k)}_\Q, \qquad
p_{(i_{1},\ldots,i_{k}) \cdots (j_1,\ldots, j_l)} : H^{\otimes n}_\Q \to \wedge^k H_\Q \otimes \cdots \otimes \wedge^l H_\Q.
\] 
For example, 
\begin{align*}
\mu_{(13)(25)} (x_1 \otimes x_2 \otimes x_3 \otimes x_4 \otimes x_5 \otimes x_6) 
&= \mu(x_1, x_3) \mu(x_2, x_5) x_4 \otimes x_6, \\
p_{(124)(35)(6)} (x_1 \otimes x_2 \otimes x_3 \otimes x_4 \otimes x_5 \otimes x_6) 
&= (x_1 \wedge x_2 \wedge x_4) \otimes (x_3 \wedge x_5) \otimes x_6 ,
\end{align*}
where $\mu : H_\Q \otimes H_\Q \to \Q$ is a natural non-degenerate antisymmetric bilinear form.
It is known that any subrepresentation of $H^{\otimes n}_\Q$ can be detected by a combination of 
these multiple contractions and projections.  
Such $\mathrm{Sp}$-projections are called {\it detectors}.
The following are detectors for this case $[21^2]$:
\begin{align*}
&p_{(123)(4)} \circ \mu_{(12)(34)} \circ \Phi_1, \quad 
p_{(123)(4)} \circ \mu_{(12)(45)} \circ \Phi_1, \quad
p_{(123)(4)} \circ \mu_{(14)(25)} \circ \Phi_1, \\
&p_{(123)(4)} \circ \mu_{(14)(56)} \circ \Phi_1, \quad
p_{(124)(3)} \circ \mu_{(14)(56)} \circ \Phi_1, \\
& p_{(123)(4)} \circ \mu_{(12)(35)} \circ \Phi_2, \quad
p_{(123)(4)} \circ \mu_{(12)(56)} \circ \Phi_2 , 
\end{align*}
by using projections 
\begin{align*}
\Phi_1 &: \mathfrak{m}_{g,1}(1)\otimes \mathfrak{m}_{g,1}(3) \oplus \wedge^2 \mathfrak{m}_{g,1}(2) 
\xrightarrow{pr_1} \mathfrak{m}_{g,1}(1)\otimes \mathfrak{m}_{g,1}(3) \subset H_\Q^{\otimes 8}, \\
\Phi_2 & : \mathfrak{m}_{g,1}(1)\otimes \mathfrak{m}_{g,1}(3) \oplus \wedge^2 \mathfrak{m}_{g,1}(2)
\xrightarrow{pr_2} \wedge^2 \mathfrak{m}_{g,1}(2) \subset H_\Q^{\otimes 8}, 
\end{align*}
where $pr_i : A_1 \oplus A_2 \to A_i$ is the projection to the $i$th component. 
Next, we take the following $7$ elements $v_j$ of 
$\wedge^2\mathfrak{m}_{g,1}(1)\otimes \mathfrak{m}_{g,1}(2)$. 
In general, $\mathfrak{m}_{g,1}(k)$ is contained in $ \mathfrak{h}_{g,1}(k)~ (\subset H^{\otimes (k+2)}_\Q)$. 
Furthermore, any element of $ \mathfrak{h}_{g,1}(k)$ can be expressed by using a Lie spider with $(k+2)$ legs 
\begin{align*}
&S(u_1, u_2, u_3, \ldots, u_{k+2}) =\\
&u_1 \otimes [u_2, [u_3,[\cdots [ u_{k+1},u_{k+2}]\cdots]]] 
+ u_2 \otimes [[u_3, [u_4,[\cdots [ u_{k+1},u_{k+2}]\cdots]]],u_1] \\
&+u_3 \otimes [[u_4, [u_5,[\cdots [ u_{k+1},u_{k+2}]\cdots]]],[u_1,u_2]] 
+ \cdots + u_{k+2} \otimes [[[\cdots [ u_1,u_2],\cdots],u_k],u_{k+1}] ,
\end{align*}
where $u_i \in H_\Q$. We fix  a symplectic basis $\{a_1, b_1, \ldots, a_g,b_g \}$ of $H_\Q$ 
with respect to $\mu$ so that  
\[
\mu(a_i,a_j) = \mu(b_i,b_j) = 0, \quad 
\mu(a_i, b_j) = \delta_{i,j} .
\]
Then we write $v_j \in \wedge^2\mathfrak{m}_{g,1}(1)\otimes \mathfrak{m}_{g,1}(2)$ as Lie spiders 
\begin{align*}
&v_1 = S(a_1,a_2,a_4) \wedge S(a_1,b_4,a_5) \otimes S(a_3,b_5,a_6,b_6), \\
&v_2 = S(a_1,a_2,a_4) \wedge S(a_1,a_5,b_5) \otimes S(a_3,b_4,a_6,b_6), \\ 
&v_3 = S(a_1,a_2,a_4) \wedge S(a_1,a_5,a_6) \otimes S(a_3,b_4,b_5,b_6), \\
&v_4 = S(a_1,a_2,a_4) \wedge S(a_3,b_4,a_5) \otimes S(a_1,b_5,a_6,b_6), \\
&v_5 = S(a_1,a_2,a_4) \wedge S(b_4,a_5,b_5) \otimes S(a_1,a_3,a_6,b_6), \\
&v_6 = S(a_1,a_2,a_4) \wedge S(b_4,a_5,a_6) \otimes S(a_1,a_3,b_5,b_6), \\
&v_7 = S(a_1,a_2,a_4) \wedge S(a_5,b_5,a_6) \otimes S(a_1,a_3,b_4,b_6).
\end{align*}
Finally, we computed the boundary operator $\partial$ on these vectors $v_j$ and 
applied the $7$ detectors $D_i$. 
Then we checked that the rank of the matrix $D_i(\partial (v_j))$ is $7$. 
This means that $\partial (v_j)$'s generate the space of the highest weight vectors 
corresponding to $7[21^2]$ in $Z_2(4)$. 
This is the process of proving that the $[21^2]$ component of $H_2(\mathfrak{m}_{g,1})_4$ vanishes.

In this way, we checked that all the $\mathrm{Sp}$-irreducible components of
$Z_2(4)$ are boundaries. This finishes the proof of 
$H_2(\mathfrak{m}_{g,1})_4=0$ and hence $H_2(\mathfrak{m}_{g})_4=0$.

\begin{remark}
All the other data are shown in the web \cite{szkHP}. 
\end{remark}

\section{Proof of $H_2(\mathfrak{m}_g)_5=0$}\label{sec:w5}

In this section, we prove the following.
\begin{prop}
$H_2(\mathfrak{m}_g)_5\cong H_2(\mathfrak{m}_{g,1})_5=0.$
\label{prop:w5}
\end{prop}

By equality \eqref{eq:hw}, 
we have
$$
H_2(\mathfrak{m}_{g,1})_5=\frac{Z_2(5)}{B_2(5)}
$$
where
$$
Z_2(5)=
\mathrm{Ker}\left((\mathfrak{m}_{g,1}(1)\otimes \mathfrak{m}_{g,1}(4))\oplus 
(\mathfrak{m}_{g,1}(2)\otimes \mathfrak{m}_{g,1}(3))
\twoheadrightarrow \mathfrak{m}_{g,1}(5)\right)
$$
and
\begin{align*}
B_2(5)=
\mathrm{Im}
&(\wedge^2\mathfrak{m}_{g,1}(1)\otimes \mathfrak{m}_{g,1}(3))
\oplus (\mathfrak{m}_{g,1}(1)\otimes  \wedge^2 \mathfrak{m}_{g,1}(2)) \\
&
\overset{\partial}{\rightarrow}
(\mathfrak{m}_{g,1}(1)\otimes \mathfrak{m}_{g,1}(4))\oplus 
(\mathfrak{m}_{g,1}(2)\otimes \mathfrak{m}_{g,1}(3))
).
\end{align*}
Here the boundary operator
\begin{equation}
\begin{split}
\partial :(\wedge^2\mathfrak{m}_{g,1}(1)\otimes \mathfrak{m}_{g,1}(3))
& \oplus (\mathfrak{m}_{g,1}(1)\otimes  \wedge^2 \mathfrak{m}_{g,1}(2)) \\
&\rightarrow 
(\mathfrak{m}_{g,1}(1)\otimes \mathfrak{m}_{g,1}(4))\oplus 
(\mathfrak{m}_{g,1}(2)\otimes \mathfrak{m}_{g,1}(3))
\label{eq:b5}
\end{split}
\end{equation}
is given by
\begin{align*}
&\wedge^2\mathfrak{m}_{g,1}(1)\otimes \mathfrak{m}_{g,1}(3)\ni (u\land v)\otimes w\longmapsto\\
&(u\otimes [v,w]-v\otimes [u,w],-[u,v]\otimes w)\in (\mathfrak{m}_{g,1}(1)\otimes \mathfrak{m}_{g,1}(4))\oplus 
(\mathfrak{m}_{g,1}(2)\otimes \mathfrak{m}_{g,1}(3))\\
&\mathfrak{m}_{g,1}(1)\otimes  \wedge^2 \mathfrak{m}_{g,1}(2)\ni u\otimes (v\land w)\longmapsto\\
&(u\otimes [v,w],-v\otimes [u,w]+w\otimes [u,v])\in (\mathfrak{m}_{g,1}(1)\otimes \mathfrak{m}_{g,1}(4))\oplus 
(\mathfrak{m}_{g,1}(2)\otimes \mathfrak{m}_{g,1}(3)).
\end{align*}
We can write 
\[
H_2(\mathfrak{m}_{g,1})_5\cong 
\mathrm{Coker}\, \left(
(\wedge^2\mathfrak{m}_{g,1}(1)\otimes \mathfrak{m}_{g,1}(3))
\oplus (\mathfrak{m}_{g,1}(1)\otimes  \wedge^2 \mathfrak{m}_{g,1}(2))
\overset{\partial}{\rightarrow}
Z_2(5) \right).
\]
As in the preceding section, 
by using Table \ref{tab:6} and applying our techniques described in \cite{mss2},
we have determined the space $Z_2(5)$ of $2$-cycles for the weight $5$ homology group
$H_2(\mathfrak{m}_{g,1})_5$, see Table \ref{tab:z25}.
\begin{table}[h]
\caption{$\text{$\mathrm{Sp}$-irreducible decomposition of $Z_2(5)$}$}
\begin{center}
\begin{tabular}{|c|l|}
\noalign{\hrule height0.8pt}
\hfil $Z_2(5)$ & $2[531]2[521^2][432]2[431^2]2[42^21]3[421^3][41^5]2[3^3]2[3^221][32^3]$    \\
{} & $3[32^21^2][321^4][31^6][2^31^3]\ 3[52][51^2]4[43]10[421]8[41^3]4[3^21]8[32^2]$\\
{} & $12[321^2]8[31^4]6[2^31]3[2^21^3]2[21^5]\ 9[41]12[32]23[31^2]12[2^21]13[21^3]$\\
{} & $11[3]17[21]6[1^3]\ 4[1]$\\
\hline
$\mathfrak{m}_{g,1}(1)\otimes \mathfrak{m}_{g,1}(4)$ &  
$[531][521^2][431^2][42^21]2[421^3][41^5][3^3][3^221][32^3]2[32^21^2][321^4][31^6][2^31^3]$    \\
{} & $2[52][51^2]2[43]6[421]5[41^3]2[3^21]4[32^2]8[321^2]6[31^4]4[2^31]3[2^21^3]2[21^5]$\\
{} & $7[41]8[32]15[31^2]8[2^21]10[21^3]\ 8[3]12[21]5[1^3]\ 3[1]$\\
\hline
$\mathfrak{m}_{g,1}(2)\otimes \mathfrak{m}_{g,1}(3)$ & 
$[531][521^2][432][431^2][42^21][421^3][3^3][3^221][32^21^2]$\\
{} & $[52][51^2]2[43]6[421]3[41^3]3[3^21]4[32^2]5[321^2]2[31^4]2[2^31][2^21^3]$\\
{} & $4[41]6[32]10[31^2]6[2^21]5[21^3]\ 4[3]8[21]3[1^3]\ 2[1]$\\
\hline
$\mathfrak{m}_{g,1}(5)$ &  $[51^2][421][3^21][321^2][2^21^3]\ 2[41]2[32]2[31^2]2[2^21]2[21^3]\ [3]3[21]2[1^3]\ [1]$  \\
\noalign{\hrule height0.8pt}
\end{tabular}
\end{center}
\label{tab:z25}
\end{table}
Moreover, we have computed the boundary operator \eqref{eq:b5} explicitly and
checked that all the $2$-cycles ($35$-types of Young diagrams with multiplicities) 
are boundaries.
This finishes the proof of 
$H_2(\mathfrak{m}_{g,1})_5=0$ and hence $H_2(\mathfrak{m}_{g})_5=0$.

\section{Proof of $H_2(\mathfrak{m}_g)_6=0$}\label{sec:w6}

In this section, we prove the following.
\begin{prop}
$H_2(\mathfrak{m}_g)_6\cong H_2(\mathfrak{m}_{g,1})_6=0.$
\label{prop:w6}
\end{prop}

By equality \eqref{eq:hw}, 
we have
$$
H_2(\mathfrak{m}_{g,1})_6=\frac{Z_2(6)}{B_2(6)}
$$
where
$$
Z_2(6)=\mathrm{Ker}\left(
(\mathfrak{m}_{g,1}(1)\otimes \mathfrak{m}_{g,1}(5))\oplus 
(\mathfrak{m}_{g,1}(2)\otimes \mathfrak{m}_{g,1}(4))\oplus
\wedge^2\mathfrak{m}_{g,1}(3)
\rightarrow \mathfrak{m}_{g,1}(6)\right)
$$
and
\begin{align*}
B_2(6)=
\mathrm{Im}
&((\wedge^2\mathfrak{m}_{g,1}(1)\otimes \mathfrak{m}_{g,1}(4))
\oplus (\mathfrak{m}_{g,1}(1)\otimes  \mathfrak{m}_{g,1}(2)\otimes\mathfrak{m}_{g,1}(3))
\oplus
\wedge^3 \mathfrak{m}_{g,1}(2))\\
&
\overset{\partial}{\rightarrow}
(\mathfrak{m}_{g,1}(1)\otimes \mathfrak{m}_{g,1}(5))\oplus 
(\mathfrak{m}_{g,1}(2)\otimes \mathfrak{m}_{g,1}(4))\oplus
\wedge^2\mathfrak{m}_{g,1}(3)
).
\end{align*}

Here the boundary operator
\begin{equation}
\begin{split}
\partial :(\wedge^2\mathfrak{m}_{g,1}(1)&\otimes \mathfrak{m}_{g,1}(4))
\oplus (\mathfrak{m}_{g,1}(1)\otimes  \mathfrak{m}_{g,1}(2)\otimes\mathfrak{m}_{g,1}(3))
\oplus
\wedge^3 \mathfrak{m}_{g,1}(2))\\
&\rightarrow 
(\mathfrak{m}_{g,1}(1)\otimes \mathfrak{m}_{g,1}(5))\oplus 
(\mathfrak{m}_{g,1}(2)\otimes \mathfrak{m}_{g,1}(4))\oplus
\wedge^2\mathfrak{m}_{g,1}(3)
\end{split}
\label{eq:b6}
\end{equation}
is given by
\begin{align*}
&\wedge^2\mathfrak{m}_{g,1}(1)\otimes \mathfrak{m}_{g,1}(4)\ni (u\land v)\otimes w\longmapsto\\
&\hspace{3mm}(u\otimes [v,w]-v\otimes [u,w],-[u,v]\otimes w,0)
\in (\mathfrak{m}_{g,1}(1)\otimes \mathfrak{m}_{g,1}(5))\oplus 
(\mathfrak{m}_{g,1}(2)\otimes \mathfrak{m}_{g,1}(4))\oplus
\wedge^2\mathfrak{m}_{g,1}(3)\\
&(\mathfrak{m}_{g,1}(1)\otimes  \mathfrak{m}_{g,1}(2)\otimes\mathfrak{m}_{g,1}(3))
\ni u\otimes v\otimes w\longmapsto\\
&\hspace{3mm}(u\otimes [v,w],-v\otimes [u,w],- [u,v]\land w)
\in (\mathfrak{m}_{g,1}(1)\otimes \mathfrak{m}_{g,1}(5))\oplus 
(\mathfrak{m}_{g,1}(2)\otimes \mathfrak{m}_{g,1}(4))\oplus
\wedge^2\mathfrak{m}_{g,1}(3)\\
&\wedge^3 \mathfrak{m}_{g,1}(2)\ni u\land v\land w\longmapsto\\
&\hspace{3mm}(0, u\otimes [v,w]+v\otimes [w,u]+w\otimes [u,v],0)
\in (\mathfrak{m}_{g,1}(1)\otimes \mathfrak{m}_{g,1}(5))\oplus 
(\mathfrak{m}_{g,1}(2)\otimes \mathfrak{m}_{g,1}(4))\oplus
\wedge^2\mathfrak{m}_{g,1}(3).
\end{align*}

Thus we can write 
\begin{align*}
&H_2(\mathfrak{m}_{g,1})_6\cong \mathrm{Coker}\, \\
& \left(
(\wedge^2\mathfrak{m}_{g,1}(1)\otimes \mathfrak{m}_{g,1}(4))
\oplus (\mathfrak{m}_{g,1}(1)\otimes  \mathfrak{m}_{g,1}(2)\otimes\mathfrak{m}_{g,1}(3))
\oplus
\wedge^3 \mathfrak{m}_{g,1}(2))
\overset{\partial}{\rightarrow}
Z_2(6)\right).
\end{align*}
As in the preceding two sections, 
by using Table \ref{tab:6} and applying our techniques described in \cite{mss2},
we can determine the space $Z_2(6)$ of $2$-cycles for the weight $6$ homology group
$H_2(\mathfrak{m}_{g,1})_6$ as in Table \ref{tab:z26}.
\begin{table}[h]
\caption{$\text{$\mathrm{Sp}$-irreducible decomposition of $Z_2(6)$}$}
\begin{center}
\begin{tabular}{|c|l|}
\noalign{\hrule height0.8pt}
\hfil 
$Z_2(6)$ 
& $[{64}][{631}] 2[{62^2}] 2[{621^2}][{61^4}][{541}]3[{532}]3[{531^2}]4[{52^21}]4[{521^3}]2[{51^5}]3[{4^22}]$\\
{} &  $5[{431^3}]2[{4^21^2}]7[{4321}]3[{42^3}]4[{42^21^2}]3[{421^4}][{3^31}]2[{3^22^2}]5[{3^221^2}]4[{3^21^4}]$\\
{} & $3[{32^31}]3[{32^21^3}] 2[{321^5}]2[{2^5}][{2^41^2}][{2^31^4}][{2^21^6}]\ 2[{62}]3[{61^2}]6[{53}]17[{521}]$\\
{} & $10[{51^3}]2[{4^2}]20[{431}]16[{42^2}]29[{421^2}]12[{41^4}]14[{3^22}]19[{3^21^2}]25[{32^21}]25[{321^3}]$    \\
{} & $8[{31^5}]5[{2^4}]12[{2^31^2}]9[{2^21^4}]3[{21^6}]\ 2[{6}]13[{51}]34[{42}]35[{41^2}]17[{3^2}]63[{321}]42[{31^3}]$ \\
{} & $25[{2^3}]40[{2^21^2}]25[{21^4}]5[{1^6}]\ 13[{4}]52[{31}]33[{2^2}]56[{21^2}]18[{1^4}]\ 26[{2}]21[{1^2}]\ 3[{0}]$\\
\hline
$\mathfrak{m}_{g,1}(1)\otimes \mathfrak{m}_{g,1}(5)$ &  
$[62^2][621^2][61^4][532][531^2]2[52^21]2[521^3][51^5][4^22][4^21^2]3[4321]3[431^3]$    \\
{} & $[42^3]2[42^21^2]2[421^4][3^22^2]3[3^221^2]3[3^21^4]2[32^31]2[32^21^3]2[321^5][2^5]$\\
{} & $[2^41^2][2^31^4][2^21^6]\ [62]2[61^2][53]8[521]6[51^3][4^2]9[431]8[42^2]15[421^2]$ \\
{} & $7[41^4]6[3^22]12[3^21^2]13[32^21]17[321^3]6[31^5]4[2^4]8[2^31^2]8[2^21^4]3[21^6]$\\
{} & $[6]8[51]16[42]19[41^2]10[3^2]36[321]25[31^3]13[2^3]28[2^21^2]$\\
{} & $19[21^4]5[1^6]\ 8[4]28[31]23[2^2]35[21^2]15[1^4]\ 14[2]16[1^2]\ 3[0]$ \\
\hline
$\mathfrak{m}_{g,1}(2)\otimes \mathfrak{m}_{g,1}(4)$ & 
$[64][631][62^2][541][532]2[531^2][52^21][521^3]2[4^22]3[4321][431^3]2[42^3]$\\
{} & $[42^21^2][421^4][3^31]2[3^221^2][32^21^3][2^21^4][2^5]\ [62][61^2]5[53]7[521]2[51^3][4^2]$\\
{} & $9[431]6[42^2]11[421^2]3[41^4]7[3^22]5[3^21^2][32^31]10[32^21]6[321^3]2[31^5][2^4]$\\
{} & $4[2^31^2]\ [6]5[51]18[42]13[41^2]5[3^2]25[321]16[31^3]12[2^3]10[2^21^2]6[21^4]$\\
{} & $5[4]25[31]11[2^2]22[21^2]3[1^4]\ 14[2]6[1^2]$\\
\hline
$\wedge^2\mathfrak{m}_{g,1}(3)$ &  
$[621^2][532][52^21][521^3][51^5][4^21^2][4321][431^3][42^21^2][3^22^2][3^21^4]$    \\
{} & $[62]3[521]3[51^3][4^2]3[431]4[42^2]4[421^2]3[41^4][3^22]4[3^21^2]3[32^21]3[321^3][2^4][2^21^4]$\\
{} & $3[51]3[42]7[41^2]5[3^2]9[321]4[31^3][2^3]7[2^21^2]2[21^4][1^6]$\\
{} & $4[4]5[31]8[2^2]5[21^2]4[1^4]\ [2]5[1^2]\ 2[0]$\\
\hline
$\mathfrak{m}_{g,1}(6)$ &  
$[62][521][51^3][431][4^2]2[42^2][421^2][41^4]2[3^21^2][32^21][321^3][2^4]$\\
{} & $3[51]3[42]4[41^2]3[3^2]7[321]3[31^3][2^3][2^21^4]5[2^21^2]2[21^4][1^6]$\\
{} & $4[4]6[31]9[2^2]6[21^2]4[1^4]\ 3[2]6[1^2]\ 2[0]$  \\
\noalign{\hrule height0.8pt}
\end{tabular}
\end{center}
\label{tab:z26}
\end{table}

We have computed the boundary operator \eqref{eq:b6} explicitly and
checked that all the $2$-cycles ($67$-types of Young diagrams with multiplicities) 
are boundaries.
In this way, we checked that all the $\mathrm{Sp}$-irreducible components of
$Z_2(6)$ are boundaries. This finishes the proof of 
$H_2(\mathfrak{m}_{g,1})_6=0$ and hence $H_2(\mathfrak{m}_{g})_6=0$.

\begin{remark}
The size of our computer computation grows very rapidly with respect to weights
and, in particular, the weight $6$ case is approximately
$1000$ times as large as the weight $4$ case.
\end{remark}

\section{
Proofs of the main results}\label{sec:ft}

\begin{proof}[Proof of Theorem \ref{th:main}] 
By Corollary \ref{cor:h2}, we have $\mathfrak{i}_g(4)\cong H_2(\mathfrak{m}_g)_4$.
On the other hand, we have $H_2(\mathfrak{m}_g)_4=0$ by Proposition \ref{prop:w4}.
Hence, we conclude that $\mathfrak{i}_g(4)=0$, namely 
$\mathfrak{t}_g(4)\cong \mathfrak{m}_g(4)$.
Then, if we combine Proposition \ref{prop:iw} with Proposition \ref{prop:w5} and Proposition \ref{prop:w6},
we conclude that $\mathfrak{i}_g(5)=\mathfrak{i}_g(6)=0$ so that
$\mathfrak{t}_g(5)\cong \mathfrak{m}_g(5)$ and $\mathfrak{t}_g(6)\cong \mathfrak{m}_g(6)$.
This finishes the proof.
\end{proof}

\begin{proof}[Proof of Corollary \ref{cor:fi}] 
Observe first that, for any $k,\ell$ with $1\leq k\leq\ell$, the quotient group $\mathcal{M}_g(k)/\mathcal{I}_g(\ell)$ is a
finitely generated nilpotent group because it is a subgroup of $\mathcal{I}_g/\mathcal{I}_g(\ell)$ which is 
finitely generated by Johnson \cite{jfg} and nilpotent.
Hence we can consider the rational form $(\mathcal{M}_g(k)/\mathcal{I}_g(\ell))\otimes\Q$
of $\mathcal{M}_g(k)/\mathcal{I}_g(\ell)$.

Now the case $k=3$ is a direct consequence of Theorem \ref{th:hain}
combined with a result in \cite{morita89}
because of the following reason. Since $\mathcal{I}_g(2)$ is a finite index subgroup of
$\mathcal{M}_g(2)=\mathcal{K}_g$ by Johnson, 
we have a short exact sequence
$$
0\rightarrow(\mathcal{M}_g(3)/\mathcal{I}_g(3))\otimes\Q
\rightarrow(\mathcal{I}_g(2)/\mathcal{I}_g(3))\otimes \Q\rightarrow 
(\mathcal{M}_g(2)/\mathcal{M}_g(3))\otimes\Q\rightarrow 0.
$$
Here $(\mathcal{I}_g(2)/\mathcal{I}_g(3))\otimes \Q=\mathfrak{t}_g(2)$
and $(\mathcal{M}_g(2)/\mathcal{M}_g(3))\otimes \Q=\mathfrak{m}_g(2)$ by definition.
Hence we conclude $(\mathcal{M}_g(3)/\mathcal{I}_g(3))\otimes\Q\cong \mathfrak{i}_g(2)\cong \Q$.
The result follows from this because we know that the homomorphism 
$d_1:\mathcal{M}_g(3)\rightarrow \Q$ is non-trivial whereas its restriction to $\mathcal{I}_g(3)$ is trivial
as shown in \cite{morita89}.

Next, we consider the cases $k\geq 4$.
We recall here
how the non-triviality of the homomorphism $d_1:\mathcal{M}_g(k)\rightarrow \Q$
for all $k\geq 4$ follows immediately from
Hain's result that the homomorphism 
$\mathfrak{t}_g(k) \to \mathfrak{m}_g(k)$
is surjective for any $k$.
Assume that $d_1:\mathcal{M}_g(k)\rightarrow \Q$ were trivial for some $k\geq 4$
and let $m$ be the smallest such. Then consider the homomorphism
\begin{equation}
\mathfrak{t}_g(m-1)=(\mathcal{I}_g(m-1)/\mathcal{I}_g(m)) \otimes \Q
\rightarrow 
\mathfrak{m}_g(m-1)=(\mathcal{M}_g(m-1)/\mathcal{M}_g(m)) \otimes \Q.
\label{eq:m}
\end{equation}
By the assumption, the non-trivial homomorphism $d_1:\mathcal{M}_g(m-1)\rightarrow\Q$
factors through $d_1:\mathcal{M}_g(m-1)/\mathcal{M}_g(m)\rightarrow\Q$.
On the other hand, we know that the restriction of $d_1$ on $\mathcal{I}_g(k)$ is trivial
for all $k\geq 3$. We now conclude that the above homomorphism \eqref{eq:m}
is {\it not} surjective which is a contradiction.

Now the case $k=4$ follows from the fact $\mathfrak{t}_g(3)\cong \mathfrak{m}_g(3)$
proved in \cite{morita99} as follows.
We have the following two exact sequences.
$$
0\rightarrow\mathfrak{t}_g(3)=(\mathcal{I}_g(3)/\mathcal{I}_g(4))\otimes \Q\rightarrow 
(\mathcal{M}_g(3)/\mathcal{I}_g(4))\otimes\Q
\rightarrow(\mathcal{M}_g(3)/\mathcal{I}_g(3))\otimes\Q
\rightarrow 0,
$$
$$
0\rightarrow(\mathcal{M}_g(4)/\mathcal{I}_g(4))\otimes\Q
\rightarrow(\mathcal{M}_g(3)/\mathcal{I}_g(4))\otimes \Q\rightarrow 
\mathfrak{m}_g(3)=(\mathcal{M}_g(3)/\mathcal{M}_g(4))\otimes\Q\rightarrow 0.
$$
By the case $k=3$ above, we have $(\mathcal{M}_g(3)/\mathcal{I}_g(3))\otimes\Q\cong\Q$.
If we put this to the first exact sequence, we obtain
$$
\mathrm{rank}\, (\mathcal{M}_g(3)/\mathcal{I}_g(4))\otimes \Q=\dim \mathfrak{t}_g(3)+1.
$$
Here $\mathrm{rank}\, (\mathcal{M}_g(3)/\mathcal{I}_g(4))\otimes\Q$ means the rank of the
nilpotent group $(\mathcal{M}_g(3)/\mathcal{I}_g(4))\otimes\Q$ over $\Q$.
On the other hand, from the second exact sequence, we have
$$
\mathrm{rank}\, (\mathcal{M}_g(3)/\mathcal{I}_g(4))\otimes \Q=\dim \mathfrak{m}_g(3)+
\mathrm{rank}\, (\mathcal{M}_g(4)/\mathcal{I}_g(4))\otimes\Q.
$$
Since $\mathfrak{t}_g(3)\cong \mathfrak{m}_g(3)$, we conclude that
$$
\mathrm{rank}\, (\mathcal{M}_g(4)/\mathcal{I}_g(4))\otimes\Q=1
$$
finishing the proof of the case $k=4$. 

The remaining cases $k=5,6,7$ follow from similar arguments as above 
by using Theorem \ref{th:main}.
\end{proof}

Next we prove 
Theorem \ref{th:kab}, Corollary \ref{cor:ki}, Theorem \ref{th:fti} and its refinements. 
For that, we recall a few facts about the relation between
the Torelli group and homology spheres.
Let $S^1\times D^2$ denote a framed solid torus and let
$H_g=\natural^g (S^1\times D^2)$ (boundary connected sum of $g$-copies of $S^1\times D^2$) denote a handlebody of genus $g$. We identify $\partial H_g$ with $\Sigma_g$ equipped with a system of $g$ meridians and longitudes.
Let $\iota_g\in \mathcal{M}_g$ be the mapping class which exchanges each meridian and longitude curves
so that the manifold $H_g\cup_{\iota_g} -H_g$ obtained by
identifying the boundaries of $H_g$ and $-H_g$ by $\iota_g$ is $S^3$.
Now for each element $\varphi\in\mathcal{I}_g$, we consider the manifold
$M_\varphi=H_g\cup_{\iota_g\varphi} -H_g$ which is a homology $3$-sphere.
It was shown in \cite{morita89} that any homology sphere can be expressed as
$M_\varphi$ for some $\varphi\in\mathcal{K}_g=\mathcal{M}_g(2)$ and 
Pitsch \cite{pitsch} further proved that $\varphi$ can be taken in $\mathcal{M}_g(3)$.

Now we recall the relation between the Casson invariant $\lambda$ and the structure of the Torelli
group as revealed in \cite{morita89} briefly (see \cite{morita91} for further results).
We defined $\lambda^*:\mathcal{K}_g\rightarrow\Z$ by setting
$\lambda^*(\varphi)=\lambda(M_{\varphi})\ (\varphi\in\mathcal{K}_g)$ and then proved
that it is a homomorphism. On the other hand, we have the following two abelian quotients of the group
$\mathcal{K}_g$
\begin{align*}
\tau_g(2)&: \mathcal{K}_g\rightarrow\mathfrak{h}_g(2)\\
d_1&: \mathcal{K}_g\rightarrow\Z
\end{align*}
where the first one is the second Johnson homomorphism and the second one is constructed in the 
above cited paper. Then we have the following.
\begin{thm}[\cite{morita89}]
The homomorphism $\lambda^*:\mathcal{K}_g\rightarrow\Z$ is expressed as
$$
\lambda^*=\frac{1}{24} d_1+\bar{\tau}_g(2)
$$
where $\bar{\tau}_g(2)$ denotes a certain quotient of the second Johnson homomorphism.
Furthermore, the restriction of $\lambda^*$ to the subgroup $\mathcal{M}_g(3)\subset \mathcal{K}_g$
is given by
$$
\lambda^*=\frac{1}{24} d_1: \mathcal{M}_g(3)\rightarrow\Z.
$$
\label{th:core}
\end{thm}

Ohtsuki initiated a theory of finite type invariants for homology $3$-spheres in \cite{ohtsuki} and
in \cite{ohtsukip} he constructed a series of such invariants $\lambda_k\ (k=1,2,\ldots)$
the first one being ($6$ times) the Casson invariant. They are now called the Ohtsuki invariants.
Garoufalidis and Levine \cite{gl} studied the relation between the finite type invariants of
homology spheres and the structure of the Torelli group, particularly its lower central series.
This work extended the case of the Casson invariant mentioned above extensively.

Now, let $v$ be a rational invariant of homology spheres of finite type $k$. Then we define
a mapping 
$$
v^*: \mathcal{I}_g\rightarrow \Q
$$
by setting $v^*(\varphi)=v(M_\varphi)$. By a result of Garoufalidis and Levine \cite{gl}, it vanishes
on $\mathcal{I}_g(k+1)$. 
On the other hand, the following result is known.

\begin{prop}[Levine {\cite[Lemma 5.5]{levine}}]
Let $v$ be an invariant of homology spheres of finite type $k$. 
Then for any $\varphi\in\mathcal{I}_g(k_1), \psi\in\mathcal{I}_g(k_2)$
with $k_1+k_2>k$, the equality
$$
v(M_{\varphi\psi})=v(M_{\varphi})+v(M_{\psi})
$$
holds. 
\end{prop}

As a direct corollary, we obtain the following. 
\begin{cor}
Let $v$ be a rational invariant of homology spheres of finite type $k$. Then
we have a well-defined map
$$
v^*: \mathcal{I}_{g} /\mathcal{I}_{g}(k+1)\rightarrow \Q
$$
and its restriction to $\mathcal{I}_{g} (m)/\mathcal{I}_{g}(k+1)$
$$
v^*: \mathcal{I}_{g}(m)/\mathcal{I}_{g}(k+1)\rightarrow \Q
$$
is a homomorphism if $2m > k$. 
\label{cor:levine}
\end{cor}

Since the Ohtsuki invariant $\lambda_k$ is of finite type $2k$, it induces a homomorphism
$$
\lambda^*_{k}: \mathcal{I}_g(k+1)/\mathcal{I}_g(2k+1)\rightarrow\Q.
$$

\begin{remark}
If we put $k=1$ here, then we obtain that
$$
\lambda^*: \mathcal{I}_g(2)/\mathcal{I}_g(3)\rightarrow\Q
$$
is a homomorphism. However, this follows from a fact already proved in \cite{morita89}
because $\mathcal{I}_g(2)$ is a finite index subgroup of $\mathcal{K}_g$
by Johnson as mentioned above.
\end{remark}

In view of Corollary \ref{cor:levine}, it should be meaningful to consider
abelian quotients of the group $\mathcal{I}_g(k)$. Here we recall known abelian
quotients of a larger group $\mathcal{M}_g(k) \supset \mathcal{I}_g(k)$.
First, we have the $k$-th Johnson homomorphism
$$
\tau_g(k): \mathcal{M}_g(k)\rightarrow \mathfrak{h}_g(k)
$$
and secondly we have its lift
\begin{equation}
\tilde{\tau}_g(k):\mathcal{M}_g(k)\rightarrow \bigoplus_{i=k}^{2k-1} \mathfrak{h}_g(i)
\label{eq:ttau}
\end{equation}
defined as follows. In \cite{morita93}, a homomorphism
$$
\tilde{\tau}_{g,1}(k): \mathcal{M}_{g,1}(k)\rightarrow H_3(N_k(\pi_1\Sigma_g^0))\quad
(\Sigma_g^0=\Sigma_g\setminus\mathrm{Int}\, D^2)
$$
was defined which is a refinement of $\tau_{g,1}(k)$.
Heap \cite{heap} studied this homomorphism by giving a geometric construction of it
and, in particular, proved that $\mathrm{Ker}\,\tilde{\tau}_{g,1}(k)=\mathcal{M}_{g,1}(2k)$. 
Comparing his result with the description of $H_3(N_k (\pi_1 \Sigma_g^0))$ by Igusa-Orr \cite{igusaorr}, we have an embedding  
$$
\tilde{\tau}_{g,1}(k): \mathcal{M}_{g,1}(k)/\mathcal{M}_{g,1}(2k)\hookrightarrow
\bigoplus_{i=k}^{2k-1} \mathfrak{h}_{g,1}(i),
$$
though the direct sum decomposition is not canonical except for the lowest part $i=k$. 
Massuyeau \cite{massuyeau} (see also Habiro-Massuyeau \cite{habiromassuyeau}) 
further studied this homomorphism by an infinitesimal approach. 
The above homomorphism \eqref{eq:ttau} is obtained from this by passing from
$\mathcal{M}_{g,1}, \mathfrak{h}_{g,1}$ to $\mathcal{M}_{g}, \mathfrak{h}_{g}$
along the lines described in \cite{morita99}\cite{mss4}.
Here, if we add the homomorphism $d_1$ and modding out the smaller subgroup
$\mathcal{I}_g(2k)\subset \mathcal{M}_{g}(2k)$, then we obtain a homomorphism 
\begin{equation}
(d_1,\tilde{\tau}_{g}(k)): \mathcal{M}_{g}(k)/\mathcal{I}_{g}(2k)\rightarrow
\Z\oplus \bigoplus_{i=k}^{2k-1} \mathfrak{h}_{g}(i)\quad (k\geq 2)
\label{eq:dttau}
\end{equation}
which is conjecturally an embedding modulo torsion elements (see Remark \ref{rem:aftercor16}). 
As far as the authors understand, this homomorphism $(d_1,\tilde{\tau}_{g}(k))$ gives the known
largest free abelian quotient of the group $\mathcal{M}_{g}(k)$ for $k\geq 2$.
The case $k=2$ gives an abelian quotient
$$
(d_1,\tilde{\tau}_g(2)):\mathcal{M}_g(2)=\mathcal{K}_g\rightarrow \Z\oplus \mathfrak{h}_{g}(2)\oplus\mathfrak{h}_{g}(3)
=\Z\oplus [2^2]\oplus [31^2]
$$
which is rationally surjective. 

Before considering the cases of $k=3,4$, here we prove Theorem \ref{th:kab}
which is equivalent to the statement that the above homomorphism gives the whole rational 
abelianization of $\mathcal{K}_g$.

\begin{proof}[Proof of Theorem \ref{th:kab}] 
Dimca, Hain and Papadima \cite[Theorem C]{dhp} proved that there is an isomorphim
$$
H_1(\mathcal{K}_g;\Q)\cong H_1([\mathrm{Gr}\,\mathfrak{t}_g,\mathrm{Gr}\,\mathfrak{t}_g]).
$$
Since the projection to the degree $1$ part
$
\mathrm{Gr}\,\mathfrak{t}_g\rightarrow \mathfrak{t}_g(1)=[1^3]
$
gives the abelianization of $\mathrm{Gr}\,\mathfrak{t}_g$,
we have
$$
[\mathrm{Gr}\,\mathfrak{t}_g,\mathrm{Gr}\,\mathfrak{t}_g]=\bigoplus_{k=2}^\infty \mathfrak{t}_g(k).
$$
Hence we can write
$$
H_1([\mathrm{Gr}\,\mathfrak{t}_g,\mathrm{Gr}\,\mathfrak{t}_g])=\bigoplus_{k=2}^\infty
H_1([\mathrm{Gr}\,\mathfrak{t}_g,\mathrm{Gr}\,\mathfrak{t}_g])_k.
$$
The results of \cite{hain} and \cite{morita99} imply that
\begin{align*}
H_1([\mathrm{Gr}\,\mathfrak{t}_g,\mathrm{Gr}\,\mathfrak{t}_g])_2&=\mathfrak{t}_g(2)\cong \Q\oplus [2^2]\\
H_1([\mathrm{Gr}\,\mathfrak{t}_g,\mathrm{Gr}\,\mathfrak{t}_g])_3&=\mathfrak{t}_g(3)\cong\mathfrak{m}_g(3)\cong [31^2].
\end{align*}
By the definition of the first homology group of Lie algebras, we can write
$$
H_1([\mathrm{Gr}\,\mathfrak{t}_g,\mathrm{Gr}\,\mathfrak{t}_g])_4
=\mathrm{Coker}\left(\wedge^2 \mathfrak{t}_g(2)\overset{[\ ,\ ]}{\rightarrow} \mathfrak{t}_g(4)\right).
$$
On the other hand, it was proved in \cite{sakasai2} that the homomorphism
$$
[\ ,\ ]: \wedge^2 \mathfrak{m}_g(2)\rightarrow \mathfrak{m}_g(4)
$$ 
is surjective. Here we use the case $k=4$ of Theorem \ref{th:main},
namely the fact that $\mathfrak{t}_g(4)\cong \mathfrak{m}_g(4)$. 
This is the key point of our proof of Theorem \ref{th:kab}. 
Then we conclude that the homomorphism 
$[\ ,\ ]: \wedge^2 \mathfrak{t}_g(2)\rightarrow \mathfrak{t}_g(4)$ is also surjective. It follows that
$$
H_1([\mathrm{Gr}\,\mathfrak{t}_g,\mathrm{Gr}\,\mathfrak{t}_g])_4=0.
$$
To finish the proof, it remains to prove that
$$
H_1([\mathrm{Gr}\,\mathfrak{t}_g,\mathrm{Gr}\,\mathfrak{t}_g])_k=0
$$
for all $k\geq 5$. In other words, we have to prove that the homomorphism
\begin{equation}
\bigoplus_{i+j=k, i, j>1} \mathfrak{t}_g(i)\otimes \mathfrak{t}_g(j)\rightarrow \mathfrak{t}_g(k)
\label{eq:bracket}
\end{equation}
induced by the bracket operation is surjective. Note that the above homomorphism
\eqref{eq:bracket} is surjective for $k=4$ as already mentioned above. 
We use induction on $k\geq 4$. Assuming that \eqref{eq:bracket} is surjective for $k$, 
we prove the surjectivity for $k+1$. Since the Torelli Lie algebra is generated
by its degree $1$ part $\mathfrak{t}_g(1)$, if we delete the condition $i, j >1$ on the 
left hand side of \eqref{eq:bracket}, then this map is surjective. Hence it is enough to show that
any element of the form
$$
[\alpha, \xi]\in \mathfrak{t}_g(k+1)\quad (\alpha \in \mathfrak{t}_g(1), \xi\in \mathfrak{t}_g(k))
$$
is contained in the image of \eqref{eq:bracket}. By the induction assumption, we can write
$$
\xi=\sum_{s}\, [\beta_s,\gamma_s]\quad (\beta_s \in \mathfrak{t}_g(k_s), 
\gamma_s\in \mathfrak{t}_g(k-k_s), 2\leq k_s\leq k-2).
$$
Then, by the Jacobi identity, we have
\begin{align*}
[\alpha, \xi]&=\sum_{s}\, [\alpha, [\beta_s,\gamma_s]]\\
&=-\sum_{s}\, \left([\beta_s, [\gamma_s,\alpha]]+[\gamma_s, [\alpha,\beta_s]]\right).
\end{align*}
This element is clearly contained in the image of  \eqref{eq:bracket} proving that it is 
surjective for $k+1$.
This completes the proof.
\end{proof}

\begin{remark}
Here we mention the relation between our computation and the statement
\begin{align*}
&H_1(\mathcal{K}_g;\Q)\cong \Q\oplus \bigoplus_{k=0}^\infty\mathrm{Coker} \,q_k\\
& q_k:\mathrm{Sym}^{k-1} [1^3]\otimes \wedge^3 [1^3] \rightarrow \mathrm{Sym}^{k}[1^3] \otimes [2^2]
\end{align*}
given by Dimca, Hain and Papadima \cite[Theorem B]{dhp}. The homomorphism $q_k$ is defined
by
$$
q_k(f\otimes (a\wedge b\wedge c))=fa\otimes \pi(b\wedge c)+fb\otimes \pi(c\wedge a)+fc\otimes \pi(a\wedge b)
$$
where $f\in \mathrm{Sym}^{k-1} [1^3], a,b,c\in [1^3]$
and
$\pi:\wedge^2 [1^3]\cong\wedge^2 \mathfrak{h}_g(1)\rightarrow [2^2]\cong\mathfrak{h}_g(2)$ denotes the bracket operation.

Now the factor $\Q$ is detected by $d_1$
and $\mathrm{Coker}\, q_0=[2^2]$ is detected by the second Johnson homomorphism $\tau_g(2)$.
These two summands correspond to $\mathfrak{t}_g(2)=\Q\oplus [2^2]$ determined by Hain \cite{hain}.
The homomorphism $q_1: \wedge^3 [1^3]\rightarrow [1^3]\otimes [2^2]$ appeared already in
\cite{morita99} (Proposition 6.3) and it was proved that $\mathrm{Coker}\, q_1=\mathfrak{t}_g(3)\cong [31^2]$.
Our computation $H_1([\mathrm{Gr}\,\mathfrak{t}_g,\mathrm{Gr}\,\mathfrak{t}_g])_4=0$
corresponds to the fact that the homomorphism $q_2$ is surjective, namely $\mathrm{Coker}\,q_2=0$.
Then, by the definition of the homomorphisms $q_k$ mentioned above, it is easy to see that they are surjective for 
all $k\geq 3$ as well.
\end{remark}

\begin{proof}[Proof of Corollary \ref{cor:ki}] 
$\mathrm{(i)}\ $ The result of Heap mentioned above implies that $\mathrm{Ker}\,\tilde{\tau}_g(2)=\mathcal{M}_g(4)$.
On the other hand, the case $k=4$ of Corollary \ref{cor:fi} shows that $\mathcal{I}_g(4)$
is a finite index subgroup of 
$$
\mathrm{Ker}(d_1: \mathcal{M}_g(4)\rightarrow \Q)=\mathrm{Ker}\, (d_1,\tilde{\tau}_2).
$$
Now Theorem \ref{th:kab} implies that 
$$
\mathrm{Ker}\, (d_1,\tilde{\tau}_2)/[\mathcal{K}_g,\mathcal{K}_g]\cong
\mathrm{Torsion} (H_1(\mathcal{K}_g;\Z)).
$$
According to Ershov-He \cite{eh} and Church-Ershov-Putman \cite{cep},
$\mathcal{K}_g$ is finitely generated. It follows that the torsion subgroup
$\mathrm{Torsion} (H_1(\mathcal{K}_g;\Z))$ of $H_1(\mathcal{K}_g;\Z)$ is a finite group
and hence $[\mathcal{K}_g,\mathcal{K}_g]$ is a finite index subgroup of $\mathrm{Ker}\, (d_1,\tilde{\tau}_2)$.
Thus both the groups $\mathcal{I}_g(4)$ and $[\mathcal{K}_g,\mathcal{K}_g]$ are
finite index subgroups of the same group $\mathrm{Ker}\, (d_1,\tilde{\tau}_2)$.
Therefore they are commensurable.

$\mathrm{(ii)}\ $ According to \cite{cep}, $\mathcal{I}_g(4)$ is finitely generated for $g\geq 7$.
On the other hand, $\mathcal{I}_g(4)$ is a finite index subgroup of $\mathrm{Ker}\, (d_1,\tilde{\tau}_2)$
as shown in $\mathrm{(i)}\ $. It follows that $\mathrm{Ker}\, (d_1,\tilde{\tau}_2)$ is finitely generated.
Since $[\mathcal{K}_g,\mathcal{K}_g]$ is a finite index subgroup of $\mathrm{Ker}\, (d_1,\tilde{\tau}_2)$
as above, we conclude that it is also finitely generated.

This finishes the proof.
\end{proof}

\begin{problem}
For a given $k\geq 3$,
determine whether the rationally surjective homomorphisms
\begin{align*}
(d_1,\tilde{\tau}_{g}(k))&: \mathcal{M}_{g}(k)\rightarrow
\Q\oplus\bigoplus_{i=k}^{2k-1} \mathfrak{m}_{g}(i)\\
\tilde{\tau}_{g}(k)&: \mathcal{I}_{g}(k)\rightarrow
\bigoplus_{i=k}^{2k-1} \mathfrak{m}_{g}(i)
\end{align*}
give the whole of $H_1(\mathcal{M}_{g}(k);\Q)$ and $H_1(\mathcal{I}_{g}(k);\Q)$
or not.
\end{problem}

In view of the result of \cite{cep} that $\mathcal{M}_{g}(k)$ and $\mathcal{I}_{g}(k)$
are all finitely generated for $g\geq 2k-1$, a positive solution 
to the above problem would imply that the subgroups
$[\mathcal{M}_{g}(k),\mathcal{M}_{g}(k)], [\mathcal{I}_{g}(k),\mathcal{I}_{g}(k)],
\mathcal{I}_g(2k)$ of the Torelli group $\mathcal{I}_g$ are commensurable
for $g\geq 4k-1$.
It would then follow that the groups 
$[\mathcal{M}_{g}(k),\mathcal{M}_{g}(k)], [\mathcal{I}_{g}(k),\mathcal{I}_{g}(k)]$ are
finitely generated in the same range.

Now we go back to the homomorphism \eqref{eq:dttau} for the cases $k=3,4$.
\begin{prop}
There are isomorphisms
$$
(d_1,p\circ \tilde{\tau}_g(3)): (\mathcal{M}_g(3)/\mathcal{I}_g(5))\otimes\Q\cong 
\Q\oplus\mathfrak{m}_g(3)\oplus \mathfrak{m}_g(4),
$$
$$
(d_1,q\circ \tilde{\tau}_g(4)):(\mathcal{M}_g(4)/\mathcal{I}_g(7))\otimes\Q\cong 
\Q\oplus\mathfrak{m}_g(4)\oplus \mathfrak{m}_g(5)\oplus \mathfrak{m}_g(6)
$$
where $p, q$ are the projections
$$
p:\mathfrak{m}_g(3)\oplus \mathfrak{m}_g(4)\oplus \mathfrak{m}_g(5)\rightarrow
\mathfrak{m}_g(3)\oplus \mathfrak{m}_g(4),
$$
$$
q:\mathfrak{m}_g(4)\oplus \mathfrak{m}_g(5)\oplus \mathfrak{m}_g(6)\oplus \mathfrak{m}_g(7)\rightarrow
\mathfrak{m}_g(4)\oplus \mathfrak{m}_g(5)\oplus \mathfrak{m}_g(6).
$$
\label{prop:tildet}
\end{prop}
\begin{proof}
This follows by combining the facts that 
$\mathcal{I}_g(5), \mathcal{I}_g(7)$ are
finite index subgroups of the kernel of the homomorphism
$d_1$ on $\mathcal{M}_g(5), \mathcal{M}_g(7)$, respectively
(see Corollary \ref{cor:fi})
and the above homomorphism \eqref{eq:dttau}.
\end{proof}

By using Corollary \ref{cor:levine} and
Proposition \ref{prop:tildet}, we obtain the following.  

\begin{thm}
Let $v$ be a rational invariant of homology spheres of finite type $4$,
including the second Ohtsuki invariant $\lambda_2$, and
let $v^*:\mathcal{I}_g\rightarrow\Q$ be the associated mapping.
Then we have the following.
\par
$\mathrm{(i)}\, $ The restriction $v^*:\mathcal{I}_g(3)\rightarrow\Q$, which is a homomorphism, 
is a quotient of the homomorphism
$$
p \circ\tilde{\tau}_g(3):\mathcal{I}_g(3)\rightarrow \bigoplus_{i=3}^{5} \mathfrak{m}_g(i)
\overset{p}{\rightarrow} \mathfrak{m}_g(3)\oplus \mathfrak{m}_g(4)
$$
where $p$ denotes the natural projection.
\par
$\mathrm{(ii)}\, $ The restriction map $v^*:\mathcal{M}_g(3)\rightarrow\Q$ factors through  
the homomorphism 
$$
(d_1,p\circ \tilde{\tau}_g(3)):\mathcal{M}_g(3)\rightarrow \Q\oplus \mathfrak{m}_g(3)\oplus \mathfrak{m}_g(4).
$$
\par
$\mathrm{(iii)}\, $ The restriction map $v^*:\mathcal{M}_g(5)\rightarrow\Q$ factors through 
the homomorphism 
$$
d_1:\mathcal{M}_g(5)\rightarrow \Q.
$$
\label{th:type4}
\end{thm}

Note that the restriction maps $v^*:\mathcal{M}_g(3)\rightarrow\Q$ and $v^*:\mathcal{M}_g(5)\rightarrow\Q$ 
might not be homomorphisms. 

\begin{thm}
Let $v$ be a rational invariant of homology spheres of finite type $6$,
including the third Ohtsuki invariant $\lambda_3$, and
let $v^*:\mathcal{I}_g\rightarrow\Q$ be the associated mapping.
Then we have the following.
\par
$\mathrm{(i)}\, $ The restriction $v^*:\mathcal{I}_g(4)\rightarrow\Q$, which is a homomorphism, 
is a quotient of the homomorphism 
$$
q \circ\tilde{\tau}_g(4):\mathcal{I}_g(4)\rightarrow \bigoplus_{i=4}^{7} \mathfrak{m}_g(i)
\overset{q}{\rightarrow} \mathfrak{m}_g(4)\oplus \mathfrak{m}_g(5)\oplus \mathfrak{m}_g(6)
$$
where $q$ denotes the natural projection.
\par
$\mathrm{(ii)}\, $ The restriction map $v^*:\mathcal{M}_g(4)\rightarrow\Q$ factors through 
the homomorphism 
$$
(d_1,q\circ \tilde{\tau}_g(4)):\mathcal{M}_g(4)\rightarrow \Q\oplus \mathfrak{m}_g(4)\oplus \mathfrak{m}_g(5)
\oplus \mathfrak{m}_g(6).
$$
\par
$\mathrm{(iii)}\, $ The restriction map $v^*:\mathcal{M}_g(7)\rightarrow\Q$ factors through 
the homomorphism 
$$
d_1:\mathcal{M}_g(7)\rightarrow \Q.
$$
\label{th:type6}
\end{thm}

Note also that the restriction maps 
$v^*:\mathcal{M}_g(4)\rightarrow\Q$ and $v^*:\mathcal{M}_g(7)\rightarrow\Q$ 
might not be homomorphisms. 

\begin{problem}
Determine the precise formulae for the expressions of $\lambda^*_2, \lambda^*_3$
in terms of $d_1$ on the subgroups $\mathcal{M}_g(5), \mathcal{M}_g(7)$, respectively.
Recall here that $\lambda^*=\frac{1}{24} d_1$ on $\mathcal{M}_g(3)$
(Theorem \ref{th:core}) 
and based on this, we call $d_1$ the core of the Casson invariant. It seems reasonable to
imagine that $\lambda^*_2, \lambda^*_3$ are constant times $d^2_1, d^3_1$, respectively,
including the cases where the constant vanishes.
\end{problem}

%

\vspace{3mm}
\noindent
{\bf Acknowledgements}
\hspace{7mm}
The authors would like to express their hearty thanks to \\
Richard Hain, Gw\'ena\"el Massuyeau, Masatoshi Sato and Shunsuke Tsuji 
for enlightening discussions. 
Also, they would like to thank Joan Birman, Robion Kirby and Sheila Newbery for
their great assistance and help on the notes \cite{jnote}.
Finally, they appreciate the referee's careful reading of the previous version and helpful comments. 

\bibliographystyle{amsplain}

\end{document}